\documentclass[english,11pt]{amsart}
\usepackage{graphicx}
\usepackage{amssymb}
\usepackage[T1]{fontenc}
\usepackage{latexsym}
\usepackage{xy}
\usepackage{euscript}
\usepackage{amsfonts,amsmath,mathrsfs}
\usepackage{verbatim}
\usepackage{fancyhdr}
\usepackage{dsfont}
\usepackage{listings}
\usepackage[toc,page]{appendix}
\usepackage{varwidth}
\usepackage{a4wide}
\usepackage{hyperref}
\usepackage{caption}
\usepackage{subcaption}

\def\di{\displaystyle}

\newtheorem{theorem}{Theorem}[section]
\newtheorem{lemma}[theorem]{Lemma}

\newtheorem{definition}[theorem]{Definition}

\begin{document}
\setcounter{tocdepth}{3}
\baselineskip 6mm
\date{}
\title{Non standard finite difference scheme preserving dynamical properties}
\author{Jacky Cresson$^{1,2}$}
\address{Laboratoire de Math\'ematiques Appliqu\'ees de Pau, Universit\'e de Pau et des Pays de l'Adour,\\
avenue de l'Universit\'e, BP 1155, 64013 Pau Cedex, France}
\address{SYRTE, Observatoire de Paris, 77 avenue Denfert-Rochereau, 75014 Paris, France}
\author{Fr\'ed\'eric Pierret$^2$}
\address{SYRTE, Observatoire de Paris, 77 avenue Denfert-Rochereau, 75014 Paris, France}


\begin{abstract}
We study the construction of a non-standard finite differences numerical scheme for a general class of two dimensional differential equations including several models in population dynamics using the idea of non-local approximation introduced by R. Mickens. We prove the convergence of the scheme, the unconditional, with respect to the discretisation parameter, preservation of the fixed points of the continuous system and the preservation of their stability nature. Several numerical examples are given and comparison with usual numerical scheme (Euler, Runge-Kutta of order 2 or 4) are detailed.
\end{abstract}

\maketitle

\vskip 5mm
\begin{tiny}
\begin{enumerate}
\item {Laboratoire de Math\'ematiques Appliqu\'ees de Pau, Universit\'e de Pau et des Pays de l'Adour, avenue de l'Universit\'e, BP 1155, 64013 Pau Cedex, France}

\item {SYRTE UMR CNRS 8630, Observatoire de Paris and University Paris VI, France}
\end{enumerate}
\end{tiny}

\tableofcontents

\section{Introduction}

Differentials equations are in general difficult to solve and study. In particular, for most of them we do not know explicit solutions. As a consequence, one is lead to perform numerical experiments using some "integrators" like the Euler or Runge-Kutta numerical scheme. The construction of these methods is based on approximation theory and focus on the way to produce finite representation of functions. Although crucial to obtain good agreements between a given solution and its approximation, it is far from being sufficient. Indeed, these numerical methods produce artefacts, i.e. numerical behaviour which are not present in the given model. Examples of these artefacts are : creation of ghost equilibrium points, change in the stability nature of existing equilibrium point or destruction of domain invariance, etc. \\

These issues are of course of fundamental importance and give in fact a way to solve it. Indeed, the artefacts produced by classical numerical methods are related to the non persistence of some important features of the dynamics generated by the differential equation. In particular, the qualitative theory of differential equations is mainly concerned with invariant objects like equilibrium points and there dynamical properties like stability or instability as well as other global properties like domain invariance, variational structures, etc. As a consequence, an idea is to construct numerical scheme not focusing on approximation problems but dealing with the respect of some dynamical informations leading to what can be called {\it qualitative dynamical numerical scheme}.\\

This program was in fact mainly developed by R. Mickens in a serie of papers (see \cite{mickens1994} \cite{mickens2003}, \cite{mickens2005}). In order to distinguish the new numerical scheme from the classical one, he gives tha name of nonstandard scheme to them.\\

The aim of this paper is to introduce a nonstandard scheme concerning a class of differential equations which cover for example all the type of the so called prey-predator models. A lot of study of nonstandard scheme for prey-predator models has been done but only with specific form of the differential equations (see \cite{dimitrov2006}, \cite{dimitrov2007}, \cite{dimitrov2008}). Our results generalize the one obtained by D.T. Dimitrov and H:V: Kojouharov in \cite{dimitrov2006}.\\

The plan of the paper is as follows : \\

In section 2, we remind classical definitions about equilibrium points and their stability for discrete and continuous dynamical systems. Section 3 gives the definition of a non-standard finite difference scheme following R. Anguelov and J.M-S. Lubuma (\cite{al2000,al2001}). In section 4, we introduce the class of two dimensional differential equations that we are considering and we study the positivity and the stability of the equilibrium points of this class of differential equations. In section 5, we introduce the non-standard scheme associate to this system with results about the preservation of stability and positivity of the initial problem. In section 6, we illustrate numerically the results on different models. In Section 7 we conclude and give some perspectives.

\section{Reminder about continuous/discrete dynamical systems}

In this Section, we remind classical results about continuous and discrete dynamical systems dealing with the {\it qualitative} behaviour of ordinary differential equations which will be studied both for our class of models and their discretisation. We refer in particular to the book of S. Wiggins \cite{wiggins} for more details and proofs.

\subsection{Vector fields}

\subsubsection{Equilibrium points and stability}

Consider a general autonomous differential equation 
\begin{equation}
\label{equagen}
\frac{dx(t)}{dt} =  f(x(t)), \quad x \in \mathbb{R}^n ,
\end{equation}
where $f \in C^2(\mathbb{R}^m,\mathbb{R}^m)$ is called the {\it vector fields} associated to (\ref{equagen}). \\

An {\it equilibrium solution} of \eqref{equagen} is a point $E\in \mathbb{R}^n$ such that $f(E)=0$. We denote by $\mathcal{F}$ the set of equilibrium points of (\ref{equagen}).\\

An important issue is to be understand the dynamics of trajectories in the neighbourhood of a given equilibrium point. This is done through different notions of {\it stability}. In our model, we will use mainly the notion of {\it asymptotic stability} which is a stronger notion than the usual {\it Liapounov stability}.

\begin{definition}[Liapounov stability]
A solution $x(t)$ of (\ref{equagen}) is said to be stable if, given $\epsilon >0$, there exists $\delta=\delta (\epsilon )>0$ such that, for any other solution, $y(t)$, of (\ref{equagen}) satisfying $\parallel x(t_0)- y(t_0)  
\parallel < \delta$, then $\parallel x(t) -y(t) \parallel <\epsilon$ for $t>t_0$, $t_0 \in \mathbb{R}$.
\end{definition}

Our main concern will be {asymptotic stability}.

\begin{definition}[Asymptotic stability] A solution $x(t)$ of (\ref{equagen}) is said to be asymptotically stable if it is Liapounov stable and for any other solution, $y(t)$, of (\ref{equagen}), there exists a constant $\delta >0$ such that if $\parallel x(t_0) -y(t_0)  
\parallel < \delta$, then $\di\lim_{t\rightarrow +\infty} \parallel x(t) -y(t) \parallel =0$.
\end{definition}

For an equilibrium $E$, an important result is that asymptotic stability can be determined from the associated {\it linear system} defined by 
\begin{equation}
\frac{dy}{dt} = Df(E) y ,
\end{equation}
where $Df(E)$ is the Jacobian of $f$ evaluated at point $E$.\\

Precisely, we have (see \cite{wiggins}, Theorem 1.2.5 p.11):

\begin{theorem}
\label{stablin}
Let $E$ be an equilibrium point of (\ref{equagen}). Assume that all the eigenvalues of $Df(E)$ have negative real parts. Then the equilibrium point $E$ is asymptotically stable.
\end{theorem}

%Throughout this article, we assume that $S$ has only hyperbolic equilibrium point that is to say %$Re(\lambda)\neq0$ for all $\lambda \in \bigcup_{x^\star \in \mathcal{F}} \sigma(Df(x^\star))$.

\subsubsection{Positivity invariance}

In many applications, in particular Biology, the variables representing the dynamical evolution of the system must belong to a given domain. A classical example is given by variables associated to density of population which must stay positive during the time evolution. Such a constraint is called positivity and is defined as follows.

\begin{definition}
\label{defpos}
The system (\ref{equagen}) satisfies the positivity property if for all initial condition $x_0\in (\mathbb{R}^+)^m$ and initial time $t_0\in\mathbb{R}^+$ we have $x(t) \in (\mathbb{R}^+)^m$ for all $t\geq t_0$.
\end{definition}
 
The positivity property can be tested using the following necessary and sufficient condition (see \cite{walter} and \cite{pavel}):

\begin{theorem}
\label{condinvpos}
The set 
\begin{equation*}
K^+:=\{x=(x_1,\dots,x_m)\in\mathbb{R}^m, \ x_i\geq 0,\, i\in I\}
\end{equation*}
is invariant for (\ref{equagen}) if and only if
\begin{equation}
f_i (x) \geq  0 \qquad \textnormal{\emph{for}}\  x\in K^+\ \textnormal{\emph{such that}}\ x_i=0, \nonumber \\
\end{equation} 
for all $i\in I$. 
\end{theorem}

\subsection{Maps}

Numerical scheme define maps which can be studied as discrete dynamical systems. 

\subsubsection{Fixed points and stability}

Consider a $C^r$ ($r\geq 1$) map 
\begin{equation}
\label{map}
x\mapsto \phi (x) ,\ x\in \mathbb{R}^n .
\end{equation}
The map $\phi$ induces a {\it discrete dynamical system} defined by 
\begin{equation}
\label{discrete}
x_{n+1} =\phi (x_n ) ,\ x_n \in \mathbb{R}^n.
\end{equation}
Let $x_0 \in \mathbb{R}^n$ be given. We denote by $\phi^n =\phi \circ \dots \circ \phi$ n-times. The bi-infinite sequence $\left \{ \phi^n (x_0 ) , n\in \mathbb{Z} \right \}$ is called the {\it orbit} of $x_0$ under the map $\phi$.\\

Everything discuss for vector fields possesses a discrete analogue. In particular, equilibrium point for vector fields correspond to {\it fixed points} for maps, i.e. point $E$ such that $\phi (E)=E$. We denote by $\mathcal{F}$ the set of fixed points of (\ref{discrete}).

\begin{theorem}
Let $E$ be a fixed point of (\ref{discrete}). Assume that all the eigenvalues of the Jacobian matrix $D\phi (E)$ have moduli strictly less than one. Then the fixed point $E$ is asymptotically stable.
\end{theorem}

\subsubsection{Positivity invariance}

The positivity invariance for vector fields has also an analogue in the discrete setting :

\begin{definition}
The discrete dynamical system (\ref{discrete}) satisfies the positivity property if for all initial conditions $x_0\in (\mathbb{R}^+)^n$, we have $x_n \in (\mathbb{R}^+)^n$ for all $n\geq 0$.
\end{definition}

A necessary and sufficient condition for positivity is that $\phi (x_0) \geq 0$ for all $x_0 \geq 0$. Although simple, this condition is in general difficult to check for a given map.

\section{Reminder about non standard numerical scheme}

We suppose the whole integration occurs over an interval $[t_0,T]$ with $T\in\mathbb{R}^+$. Let $h \in \mathbb{R}$ with $h>0$. For $k\in \mathbb{N}$, we denote by $t_k$ the discrete time defined by $t_k = kh$. \\

\begin{definition}
A general one-step numerical scheme with a step size $h$, that approximates the solution of a general system such as (\ref{equagen}) can be written in the form
\begin{equation}
\label{equagenh}
\quad X_{k+1} = \phi_h (X_k)
\end{equation}
where $\phi_h$ is $C^2(\mathbb{R}^m,\mathbb{R}^m)$ and $X_k$ is the approximate solution of (\ref{equagen}) at time $t_k$, for all $k\ge 0$ and $X_0$ the initial value.
\end{definition}

\begin{definition}
A numerical method {\em converges} if the numerical solution $X_k$ satisfies
\begin{equation}
\sup_{0 \le t_k \le T}  \|X_k - x (t_k) \|_\infty \rightarrow 0 
\end{equation}
as $h \rightarrow 0$ and $X_0 \rightarrow x(t_0)$. 

It is {\em accurate of order $p$} if
\begin{equation}
\sup_{0 \le t_k \le T}  \|X_k - x(t_k) \|_\infty = O(h^p) + O(\|X_0 - x(t_0)\|_\infty )
\end{equation}
as $h \rightarrow 0$ and $X_0 \rightarrow x(t_0)$.
\end{definition}

Following R. Anguelov and J.M-S. Lubuma (see \cite{al2000,al2001}), we define the notion of {\it non-standard finite difference scheme} as follows :

\begin{definition}
A general one-step numerical scheme that approximate the solution of (\ref{equagen}) is called \textbf{N}on-\textbf{S}tandard \textbf{F}inite \textbf{D}ifference scheme if at least one of the following conditions is satisfied :
\begin{itemize}
\item $\frac{d}{dt}X(t_k)$ is approximate as $\frac{X_{k+1}-X_k}{\varphi(h)}$ where $\varphi(h)=h+O(h^2)$ a nonnegative function,
\item $\phi_h (f,X_k)=\tilde{\phi}_h (f,X_k,X_{k+1})$ is a nonlocal approximation of $f(t_k,X(t_k))$.
\end{itemize}
\end{definition}

The terminology of {\it nonlocal approximation} comes from the fact that the approximation of a given function $f$ is not only given at point $x_k$ by $f(x_k)$ but can eventually depends on more points of the orbits as for example 
\begin{align*}
x^2(t_k) &\approx x_k x_{k+1}, \ x_k x_{k-1}, \ x_k \left(\frac{x_{k-1}+x_{k+1}}{2}\right), \\
x^3(t_k) &\approx x_k^2 x_{k+1}, \ x_{k-1} x_k x_{k+1}. \\
\end{align*}
In the previous definition, we have concentrated on the easiest case, depending only on $x_k$ and $x_{k+1}$.

\section{A class of ordinary differential equations}

We consider the two dimensional system of ordinary differential equations defined for $(x,y)\in 
\mathbb{R}\times \mathbb{R}$ by
$$
\left .
\begin{array}{lll}
	\di\frac{dx}{dt} & =  x \left(f_+(x,y)-f_-(x,y) \right), \quad x(t_0)=x_0 \ge 0 ,\\
	\di\frac{dy}{dt} & =  y \left(g_+(x,y)-g_-(x,y) \right), \quad y(t_0)=y_0 \ge 0 ,
\end{array}
\right .
\eqno{(E)}
$$
where $f_+,f_-$ and $g_+,g_-$ are positive for all $(x,y) \in \mathbb{R}^+ \times \mathbb{R}^+$ and of class $C^1$. \\

The vector field associated to (E), denoted by $\phi :\mathbb{R}^2 \rightarrow \mathbb{R}^2$, is defined by 
\begin{equation}
\varphi(x,y)=\begin{pmatrix}
x \left(f_+(x,y)-f_-(x,y) \right) \\
y \left(g_+(x,y)-g_-(x,y) \right)
\end{pmatrix}
\end{equation}
Equation (E) contains classical examples like the {\it general Rosenzweig-MacArthur predator-prey model} (see \cite{brauer}, p. 182) studied in particular by D. T. Dimitrov and H. V. Kojouharov \cite{dimitrov2006}.

\subsection{Equilibrium points and stability}

The set of equilibrium points of $(E)$ is denoted by $\mathcal{F}$. By definition, a point $(x,y) \in \mathcal{F}$ satisfies
$$
\begin{array}{lll}
x \left(f_+(x,y)-f_-(x,y) \right)& = & 0, \\
y \left(g_+(x,y)-g_-(x,y) \right)& = & 0.
\end{array}
$$
Equilibrium points of $(E)$ consist in the origin $O=(0,0)$ and {\it potential} equilibrium points which can belong to three distinct family given by 

\begin{align*}
E_1 &= (x_\sharp,0) \ \text{where} \ x_\sharp\neq 0 \ \text{and} \ f_+(x_\sharp,0)=f_-(x_\sharp,0) \ , \\
E_2 &= (0,y_\sharp) \ \text{where} \ y_\sharp\neq 0 \ \text{and} \ g_+(0,y_\sharp)=g_-(0,y_\sharp) \ \\
E_3 &= (x_\star,y_\star) \ \mbox{with}\ x_\star \not=0 ,\, y_\star \not= 0 \ \text{where} \ f_+(x_\star,y_\star)=f_-(x_\star,y_\star) \ \text{and} \ g_+(x_\star,y_\star)=g_-(x_\star,y_\star) \ ,
\end{align*}
depending on the existence of solutions for each equation. The family $E_1$ and $E_2$ can naturally be included in the family $E_3$ if we allow null components. However, in many examples, only family $E_1$ and $E_2$ appear. Moreover, the preservation of a point of the family $E_3$ behaves in general very differently as the preservation of an equilibrium point of the families $E_1$ and $E_2$ (see Section \ref{stability-nsfdm}).\\

The stability/instability nature of these equilibrium points can be completely solved. Indeed, we have the following Lemmas describing explicitly the eigenvalues for each type of equilibrium point.

\begin{lemma}
\label{stabilityO}
The origin has eigenvalues given by $\lambda^0_1 = \left(f_+-f_-\right)(0,0)$ and $\lambda^0_2=\left(g_+-g_-\right)(0,0)$.
\end{lemma}

The proof is given in Section \ref{proof-stabilityO}.

\begin{lemma}
\label{stability-E1-E2}
Assume that $(E)$ possesses an equilibrium point belonging to the family $E_1$ (resp. $E_2$) denoted by $(x_\sharp,0)$ (resp. $(0,y_\sharp)$). The eigenvalues are given by $\lambda^1_1 = x_\sharp\left(\partial_x f_+ - \partial_x f_- \right)(x_\sharp,0)$ and $\lambda^1_2=\left(g_+-g_-\right)(x_\sharp,0)$ (resp. $\lambda^2_1=f_+(0,y_\sharp)-f_-(0,y_\sharp)$ and $\lambda^2_2=y_\sharp\left(\partial_y g_+ - \partial_y g_- \right)(0,y_\sharp)$).
\end{lemma}

The proof is given in Section \ref{proof-E1-E2}.\\

Finally, we have the following general result :

\begin{lemma}
\label{stability-E3}
Assume that $(E)$ possesses an equilibrium point belonging to the family $E_3$ denoted by $(x_* ,y_*)$. We denote by $T$ (resp. $D$) the trace (resp. determinant) of the Jacobian matrix of $\phi$ at point $(x_* ,y_*)$ denoted by $D\varphi(x_\star,y_\star)$ is given by 
\begin{equation}
T=x_\star\left(\partial_x f_+ - \partial_x f_- \right)(x_\star,y_\star)+y_\star\left(\partial_y g_+ - \partial_y g_- \right)(x_\star,y_\star) ,
\end{equation}
and
\begin{equation}
D=x_\star y_\star \left(\left(\partial_x f_+ - \partial_x f_- \right)\left(\partial_y g_+ - \partial_y g_- \right)-\left(\partial_y f_+ - \partial_y f_- \right)\left(\partial_x g_+ - \partial_x g_- \right)\right)(x_\star,y_\star) .
\end{equation}
If $T^2-4 D\geq 0$ then $D\varphi(x_\star,y_\star)$ has eigenvalues given by $\di\frac{1}{2} (T\pm \sqrt{T^2 -4D})$. 

Else if $T^2-4 D<0$ then $D\varphi(x_\star,y_\star)$ has eigenvalues $\di\frac{1}{2} (T\pm i\sqrt{4D -T^2})$ with $i^2 =1$. 
\end{lemma}

\begin{proof}
The Jacobian of $\Phi$ at an equilibrium point of the family $E_3$ is given by
\begin{equation}
D\varphi(x_\star,y_\star) = \begin{pmatrix}
x_\star\left(\partial_x f_+ - \partial_x f_- \right)(x_\star,y_\star) & x_\star\left(\partial_y f_+ - \partial_y f_- \right)(x_\star,y_\star) \\ 
y_\star\left(\partial_x g_+ - \partial_x g_- \right)(x_\star,y_\star) & y_\star\left(\partial_y g_+ - \partial_y g_- \right)(x_\star,y_\star)
\end{pmatrix}
.
\end{equation}
The characteristic polynomial is then given by $\lambda^2 -T \lambda +D =0$ where $T$ and $D$ correspond to the trace and determinant of $D\varphi(x_\star,y_\star)$. This concludes the proof.
\end{proof}

We denote by $(S_{E_i})$ for $i=0,1,2,3$ the conditions where $Re(\lambda^i_1)$ and $Re(\lambda^i_2)$ are strictly negatives that is to say $(S_{E_i})$ is the conditions for which the equilibrium point in ${E_i}$ is linearly asymptotically stable (and then asymptotically stable by Theorem \ref{stablin}). Using the previous Lemmas we have the following explicit characterization of linearly asymptotically stable equilibrium points in each family :

\begin{lemma}[Conditions of linear asymptotic stability] 
The conditions of asymptotic stability $(S_{E_i})$ for $i=0,1,2,3$ are given by :
\begin{itemize}
\item The origin is linearly asymptotically stable if and only if $\left(f_+-f_-\right)(0,0) <0$ and $\left(g_+-g_-\right)(0,0) <0$.

\item An equilibrium point belonging to the family $E_1$ (resp. $E_2$) denoted by $(x_\sharp,0)$ (resp. $(0,y_\sharp)$) is linearly asymptotically stable if $x_\sharp\left(\partial_x f_+ - \partial_x f_- \right)(x_\sharp,0) <0$ and $\left(g_+-g_-\right)(x_\sharp,0) <0$ (resp. $f_+(0,y_\sharp)-f_-(0,y_\sharp) <0$ and $y_\sharp\left(\partial_y g_+ - \partial_y g_- \right)(0,y_\sharp) <0$).

\item An equilibrium point of the family $E_3$ is linearly asymptotically stable if and only if $T<0$ and $D>0$.
\end{itemize}
\end{lemma}

These conditions will be used in Section \ref{stability-nsfdm}. Only the third condition is not trivial although classical. It uses the trace-determinant diagram to characterize the dynamical behaviour of linear systems (see \cite{hubbard}).  

\subsection{Positivity invariance}

Using Theorem \ref{condinvpos}, we easily derive the following result :

\begin{theorem}
The system $(E)$ satisfies the positivity property.
\end{theorem}

\begin{proof}
The conditions of Theorem \ref{condinvpos} are clearly satisfied for $(E)$.
\end{proof}

\section{A non-standard finite difference scheme}

The notion of non-standard scheme was introduced by R. E. Mickens at the end of the 80's. We refer to the book \cite{mickens1994} in particular Chapter 3 for more details and an overview of Mickens's ideas and to \cite{mickens2005adv} for a more recent presentation.

\subsection{Definition}

We introduce the following non-standard finite difference scheme :

\begin{definition}
The NSFD scheme of $(E)$ is given by
\begin{equation}
\label{nsfdm-scheme}
\begin{array}{lll} 
\di\frac{x_{k+1}-x_k}{h} &= & x_k f_+(x_k,y_k)-x_{k+1}f_-(x_k,y_k), \\
\di\frac{y_{k+1}-y_k}{h} &= & y_k g_+(x_k,y_k)-y_{k+1}g_-(x_k,y_k).
\end{array}
\end{equation}
\end{definition}

The associated discrete dynamical system is defined by the map $\varphi_{NS,h} :\mathbb{R}^2 \rightarrow \mathbb{R}^2$ given by
\begin{equation}
\varphi_{NS,h}(x_k,y_k)=\begin{pmatrix}
x_k \left( \di\frac{1 + h f_+(x_k,y_k)}{1 + h f_-(x_k,y_k)} \right) \\ 
y_k \left( \di\frac{1 + h g_+(x_k,y_k)}{1 + h g_-(x_k,y_k)} \right)
\end{pmatrix} 
\end{equation}

As usual, the main issue for numerical scheme is to prove {\it convergence}. We have the following result :

\begin{theorem}
The NSFD scheme (\ref{nsfdm-scheme}) is convergent and of order one.
\label{thmconv}
\end{theorem}

The proof is given in Appendix \ref{proof-convergence}.

\section{Dynamical properties of the NSFD scheme}

\subsection{Positivity invariance}

As $f_+,f_-$ and $g_+,g_-$ are positive for all $(x,y) \in \mathbb{R}^+ \times \mathbb{R}^+$, we have :

\begin{lemma}
The NSFD scheme preserves positivity for arbitrary $h$. \label{lempos}
\end{lemma}

\subsection{Equilibrium points}

In general we have $\mathcal{F} \subset \mathcal{F}_h$ because numerical schemes induce sometimes artificial fixed points like the Runge-Kutta methods. These points are often called extraneous or ghost fixed points (see \cite{cartwright} p. 16). The NSFD scheme behaves very nicely :

\begin{lemma}
For arbitrary $h$, we have $\mathcal{F} =\mathcal{F}_h$.
\end{lemma}

\subsection{Stability and instability}
\label{stability-nsfdm}

The stability of equilibrium point of $(E)$ under discretization will correspond to the stability of the fixed point of the map $\varphi_{NS,h}$. We have :

\begin{theorem}
The NSFD scheme preserves the stability nature of the origin and equilibrium points of type $E_1$ or $E_2$ for arbitrary $h$. 
\label{thmstab1}
\end{theorem}

For equilibrium points belonging to the family $E_3$, we have not preservation of the stability nature unconditionally with respect to the parameter $h$ but only for a sufficiently small one. 

\begin{theorem}
If $(S_{E_3})$ are satisfied then there exist a constant $C(E_3)$ such that for all $0<h<C(E_3)$, $E_3$ is a stable fixed point for $\varphi_{NS,h}$. 
\label{thmstab2}
\end{theorem}

\subsection{A remark concerning non-locality and weighted time step}

R. Mickens has derived many "tricks" in order to preserve particular dynamical behaviour. The one used in this paper is a non-local approximation of a given function. There exists also the possibility to use a weighted time step. This is done for example in \cite{dimitrov2006} where the authors mix the two tricks in order to preserve the stability/positivity for a particular case of our model. This approach gives in our case the following numerical scheme :

\begin{definition}
Let $\varphi(h)=h+O(h^2)$ be a nonnegative function. The extended NSFD scheme of $(E)$ is given by
$$
\begin{array}{lll}
\di\frac{x_{k+1}-x_k}{\varphi(h)} & = & x_k f_+(x_k,y_k)-x_{k+1}f_-(x_k,y_k), \\
\di\frac{y_{k+1}-y_k}{\varphi(h)} & = & y_k g_+(x_k,y_k)-y_{k+1}g_-(x_k,y_k) .
\end{array}
$$
\end{definition}

This scheme defines a natural map 
\begin{equation*}
\varphi_{ENS,h}(x_k,y_k)=\begin{pmatrix}
x_k \left( \frac{1 + \varphi(h) f_+(x_k,y_k)}{1 + \varphi(h) f_-(x_k,y_k)} \right) \\ 
y_k \left( \frac{1 + \varphi(h) g_+(x_k,y_k)}{1 + \varphi(h) g_-(x_k,y_k)} \right)
\end{pmatrix} 
\end{equation*}

Our results can be extended to this new numerical scheme. However, the main properties have nothing to do with the choice of a weighted time increment but are only induced by the non-local approximation.

\section{Numerical examples}

Our aim in this Section is to illustrate the advantages of the non-standard scheme with respect to other classical methods including the Euler scheme or the Runge-Kutta method of order 2 or 4 on a specific example. In particular, we provide simulations illustrating some well-known numerical artefacts produced by these methods and which are corrected by the non-standard scheme. \\

We consider the following class of model :

$$
\begin{array}{lll} 
\di\frac{dx}{dt} & = & x \left(b-\left(bx+\frac{ay}{c+x}\right) \right), \\
\di\frac{dy}{dt} & = & y \left(\frac{x}{c+x}-d \right),
\end{array}
$$
where $a,b,c,d$ are real constants.\\

We use two particular sets of values for our simulations :
\begin{itemize}
\item {\bf Model 1} : $a=2,b=1,c=0.5,d=6$.
\item {\bf Model 2} : $a=2,b=1,c=1,d=0.2$.
\end{itemize} 

\subsection{Equilibrium points artefacts}

Model 1 possesses two equilibrium points corresponding to the origin with eigenvalues $\lambda^0_1=1$, $\lambda^0_2=-6$ and an equilibrium point of type $E_1$ given by $P_1 =(1,0)$ with eigenvalues $\lambda^1_1=-1$, $\lambda^1_2=\frac{-16}{3}$.\\

In figure \ref{fig_ex1}, we provide numerical simulations for the initial conditions $x_0=15$, $y_0=0.1$. The main result is that the NSFD scheme has a better dynamical behaviour than the Runge-Kutta method of order $2$. In particular the Runge-Kutta method of order 2 produces for $h=0.1$ a virtual equilibrium point. 

\begin{figure}[h]
        \centering
        \begin{subfigure}[b]{0.45\textwidth}
                \includegraphics[width=\textwidth]{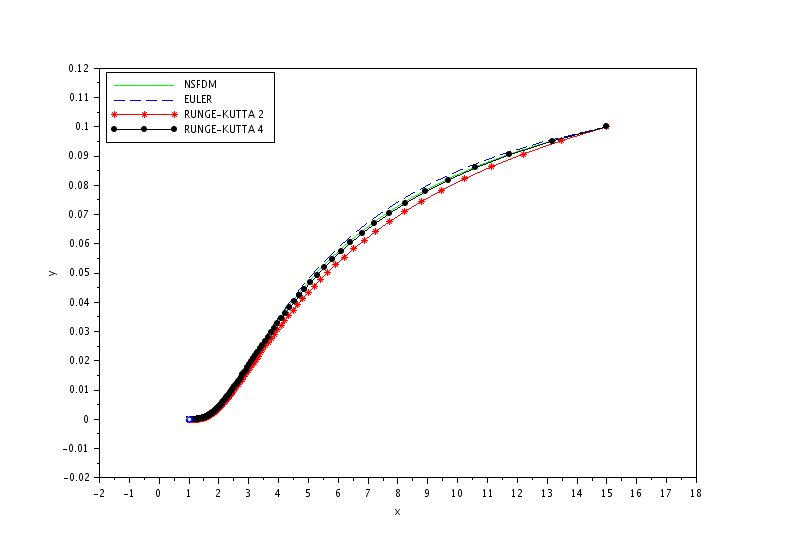}
                \caption{$h=0.01$}
        \end{subfigure}%
        ~ %add desired spacing between images, e. g. ~, \quad, \qquad, \hfill etc.
        \begin{subfigure}[b]{0.45\textwidth}
                \includegraphics[width=\textwidth]{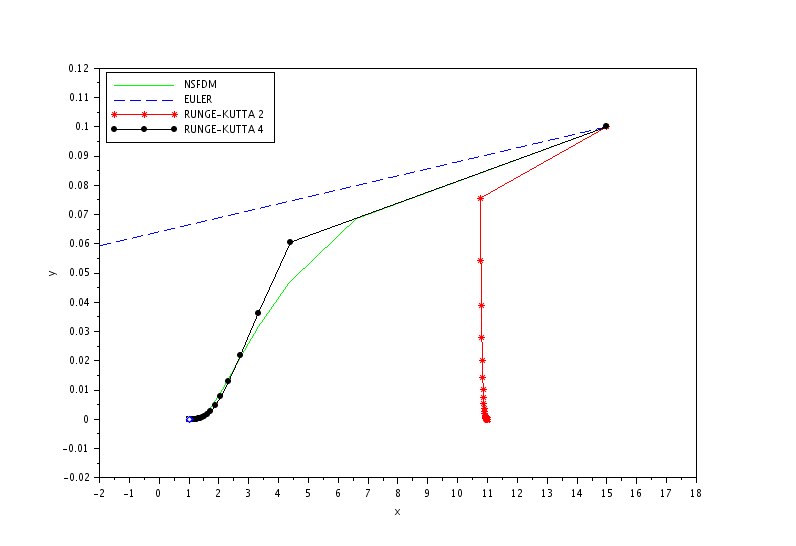}
                \caption{$h=0.1$}
        \end{subfigure}
                		
        \begin{subfigure}[b]{0.45\textwidth}
                \includegraphics[width=\textwidth]{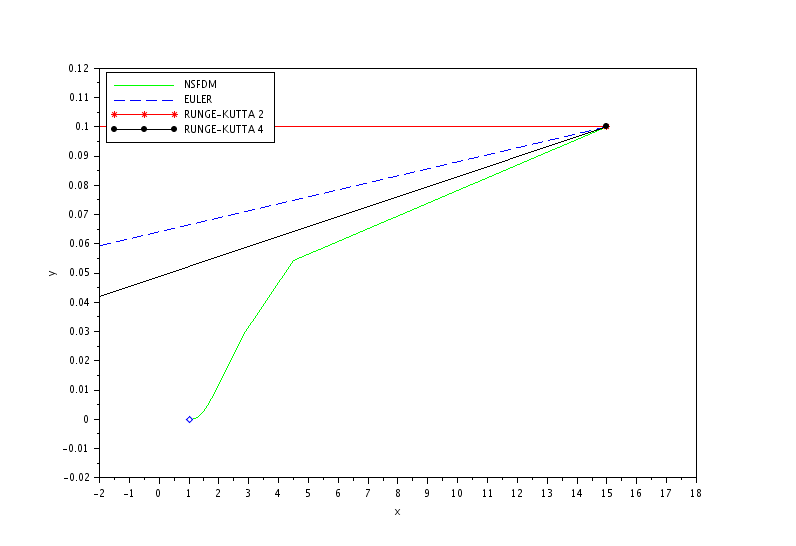}
                \caption{$h=0.2$}
        \end{subfigure}
        \caption{Numerical simulations of Example 1 with $x_0=15$, $y_0=0.1$.}\label{fig_ex1}
\end{figure}

Moreover, the NSFD scheme behaves equivalently to the Runge-Kutta method of order 4. From the computational point of view, this result is extremely strong as the algorithmic complexity of the NSFD compares to the Runge-Kutta of order 4 is very weak.

\subsection{Stability/instability artefacts}

We use again Model 1. The simulations are made with the initial condition $x_0=0.3$, $y_0=7.5$ and are given in Figure \ref{figex2}. 

\begin{figure}
        \centering
        \begin{subfigure}[b]{0.45\textwidth}
                \includegraphics[width=\textwidth]{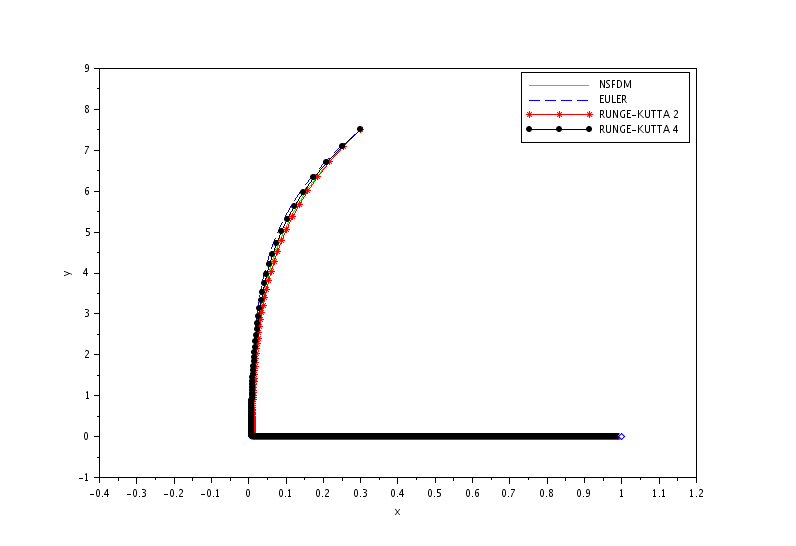}
                \caption{$h=0.01$}
        \end{subfigure}%
        ~ %add desired spacing between images, e. g. ~, \quad, \qquad, \hfill etc.
        \begin{subfigure}[b]{0.45\textwidth}
                \includegraphics[width=\textwidth]{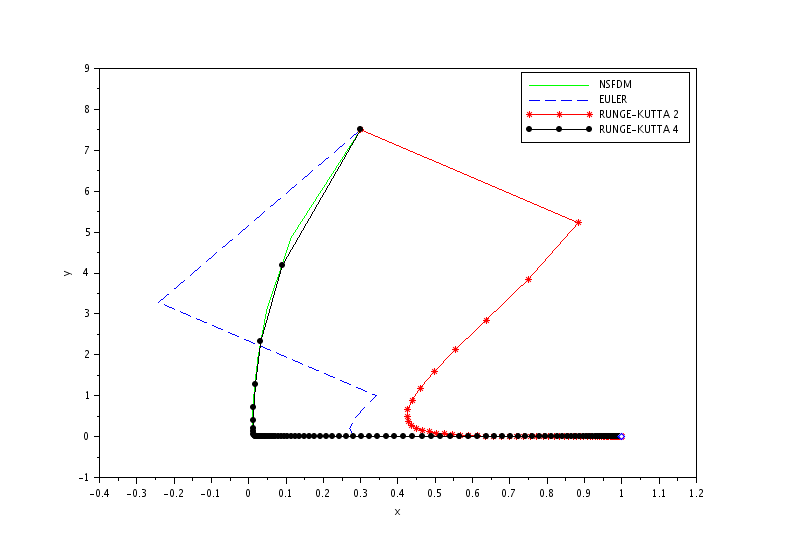}
                \caption{$h=0.1$}
        \end{subfigure}
                		
        \begin{subfigure}[b]{0.45\textwidth}
                \includegraphics[width=\textwidth]{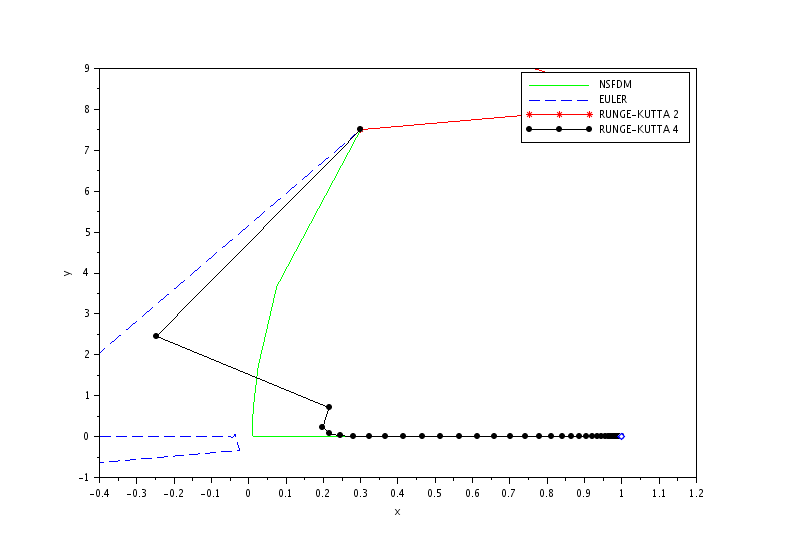}
                \caption{$h=0.2$}
        \end{subfigure}
        ~
        \begin{subfigure}[b]{0.45\textwidth}
                \includegraphics[width=\textwidth]{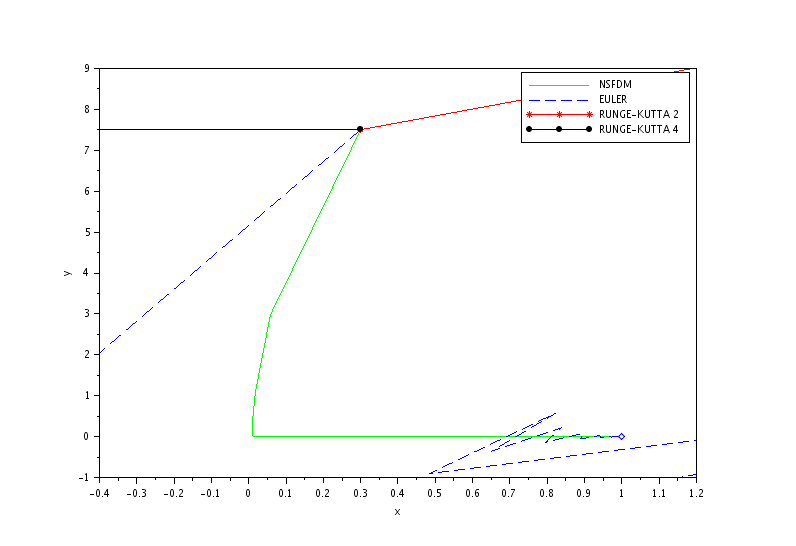}
                \caption{$h=0.3$}
        \end{subfigure}

        \begin{subfigure}[b]{0.45\textwidth}
                \includegraphics[width=\textwidth]{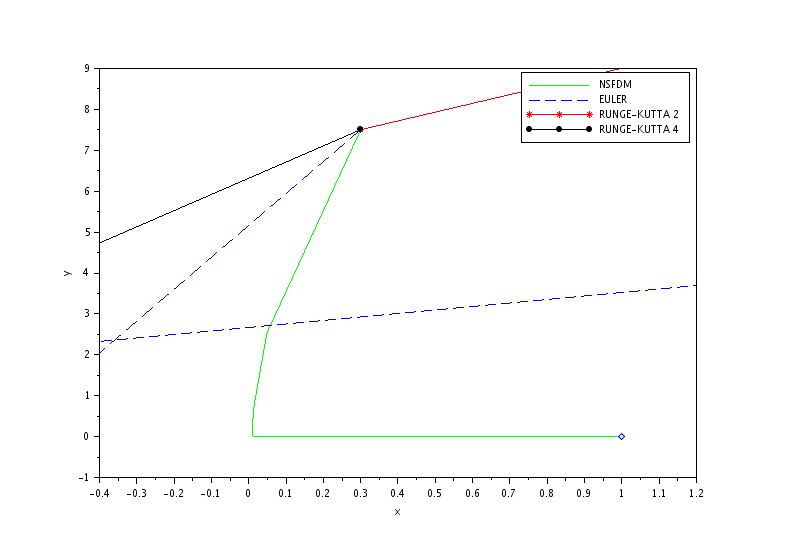}
                \caption{$h=0.4$}
        \end{subfigure}
        ~
        \begin{subfigure}[b]{0.45\textwidth}
                \includegraphics[width=\textwidth]{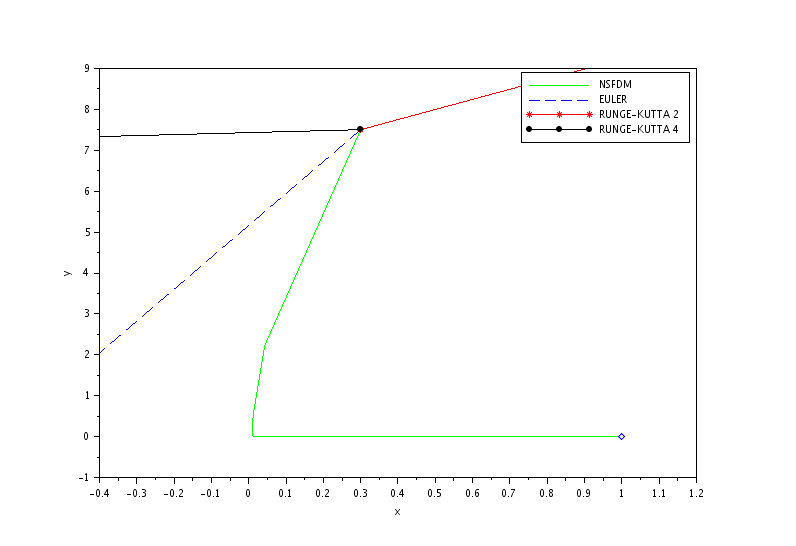}
                \caption{$h=0.5$}
        \end{subfigure}
        \caption{Numerical simulations of Example 2 with $x_0=0.3$, $y_0=7.5$.}
\label{figex2}
\end{figure}

The NSFD scheme reproduces the correct dynamical behaviour already for $h=0.5$. In the contrary, the Euler, Runge-Kutta of order 2 or 4 do not match the real dynamics for $h$ from $h=0.5$ to $h=0.3$. The Runge-Kutta of order 4 produces a better agreement for $h=0.2$ but with artificial oscillations. The correct behaviour is only recovered for $h=0.1$ for the Runge-Kutta of order 4  and $h=0.01$ for the others.\\

Another example with simulations done with initial conditions $x_0=0.4$, $y_0=0.4$ is given in Figure \ref{fig_ex3}.

\begin{figure}[h]
        \centering
        \begin{subfigure}[b]{0.45\textwidth}
                \includegraphics[width=\textwidth]{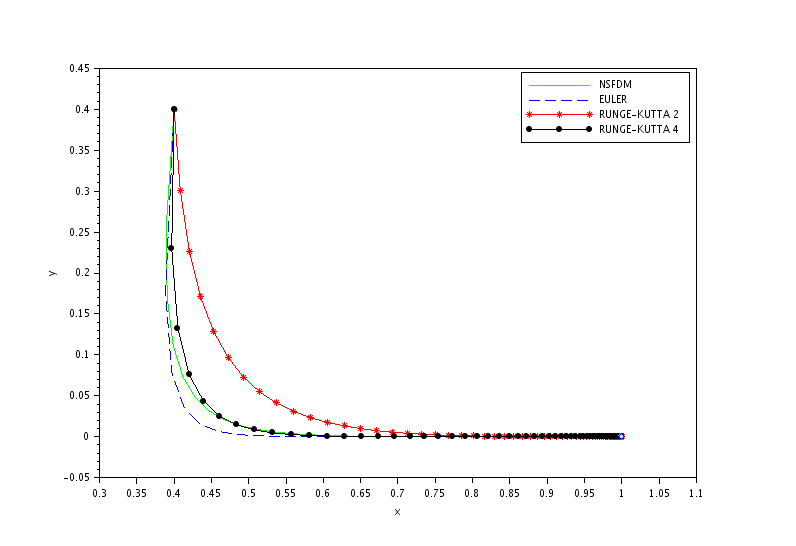}
                \caption{$h=0.1$}
        \end{subfigure}%
        ~ %add desired spacing between images, e. g. ~, \quad, \qquad, \hfill etc.
        \begin{subfigure}[b]{0.45\textwidth}
                \includegraphics[width=\textwidth]{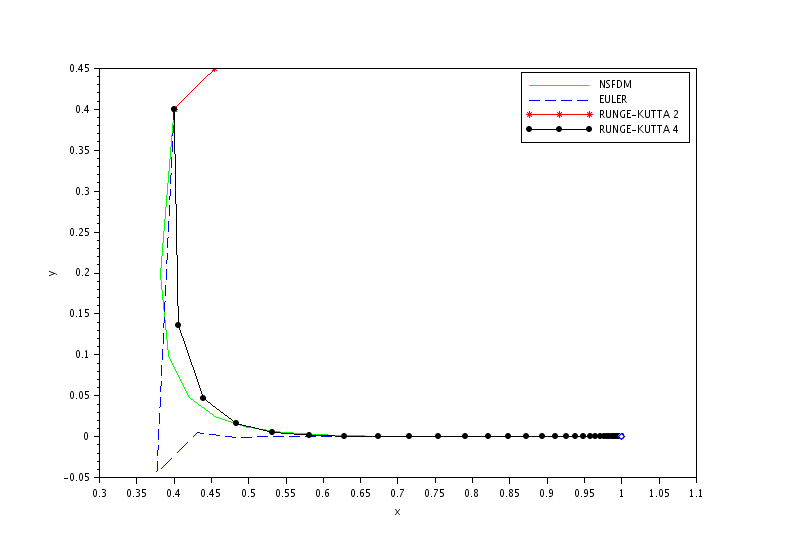}
                \caption{$h=0.2$}
        \end{subfigure}
                		
        \begin{subfigure}[b]{0.45\textwidth}
                \includegraphics[width=\textwidth]{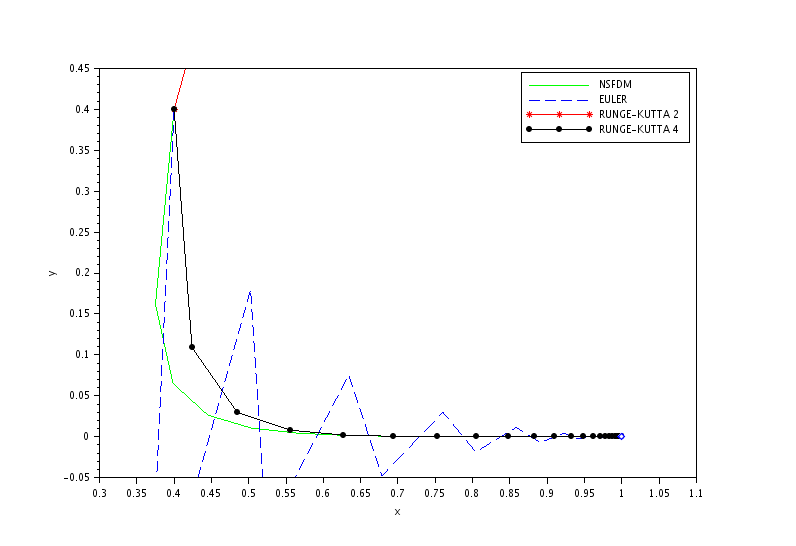}
                \caption{$h=0.3$}
        \end{subfigure}
        ~
        \begin{subfigure}[b]{0.45\textwidth}
                \includegraphics[width=\textwidth]{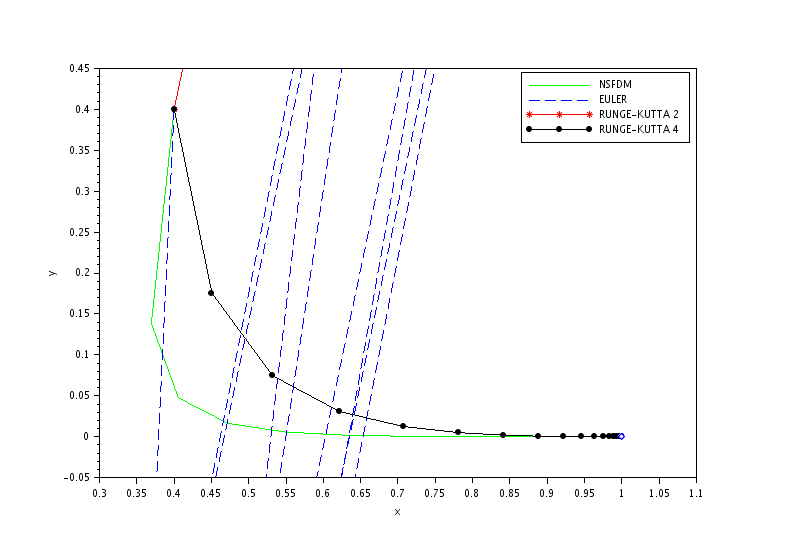}
                \caption{$h=0.4$}
        \end{subfigure}

        \begin{subfigure}[b]{0.45\textwidth}
                \includegraphics[width=\textwidth]{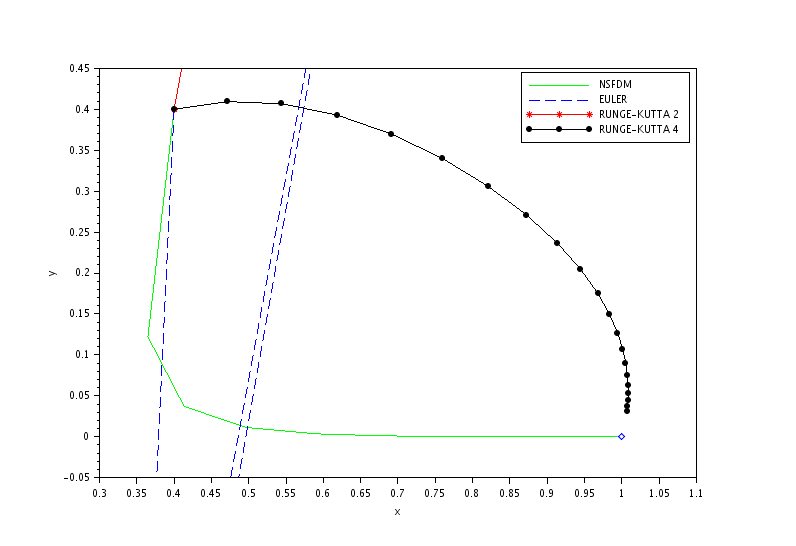}
                \caption{$h=0.5$}
        \end{subfigure}
        ~
        \begin{subfigure}[b]{0.45\textwidth}
                \includegraphics[width=\textwidth]{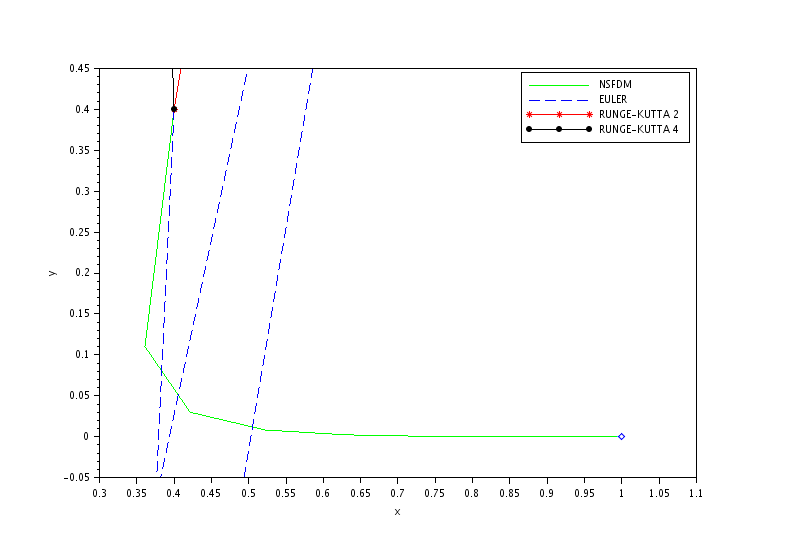}
                \caption{$h=0.6$}
        \end{subfigure}
        \caption{Numerical simulations of Example 1 with $x_0=0.4$, $y_0=0.4$.}
        \label{fig_ex3}
\end{figure}

The NSFD scheme reproduces the correct dynamical behaviour already for $h=0.6$. The Euler, Runge-Kutta of order 2 does not match the real dynamics for $h$ from $h=0.6$ to $h=0.2$. The Runge-Kutta of order 4 produces a better agreement for $h=0.4$ but with a completely different trajectory for $h=0.5$ even if the convergence to the equilibrium point $P_1$ is respected. The correct behaviour is only recovered for $h=0.1$ for the others.\\

Model 2 possesses three equilibrium points corresponding to the origin with eigenvalues $\lambda^0_1=1$, $\lambda^0_2=-\frac{1}{5}$, and one fixed point in the family $E_1$ and $E_3$ : $P_1=(1,0)$ with eigenvalues $\lambda^1_1=-1$, $\lambda^1_2=\frac{3}{10}$ and $P_3=(\frac{1}{4},\frac{15}{32})$ with eigenvalues $\lambda^3_1=\frac{1}{20} \left(-1+i \sqrt{47}\right)$, $\lambda^3_2=\frac{1}{20} \left(-1-i \sqrt{47}\right)$. 

The equilibrium point $P_3$ is stable. Theorem \ref{thmstab2} ensures that the $NSFDM$ scheme preserves the stability as long as $0<h<C=1$. However, the benefit of using the NSFD scheme is not as evident as in the previous case. Indeed, as displays in Figure \ref{fig_ex4} only the Runge-Kutta of order 4 converges to the equilibrium point for $h$ from $h=4$ to $h=0.01$. For $h$ between $h=4$ and $h=1$ the $NSFDM$ has a periodic limit cycle. The Runge-Kutta of order 2 converge to the equilibrium point from $h=2$ to $h=0.01$ but for $h=4$ it diverges. The Euler has also a periodic limit cycle limit for $h$ from $h=2$ to $h=1$.

\begin{figure}[h]
        \centering
        \begin{subfigure}[b]{0.45\textwidth}
                \includegraphics[width=\textwidth]{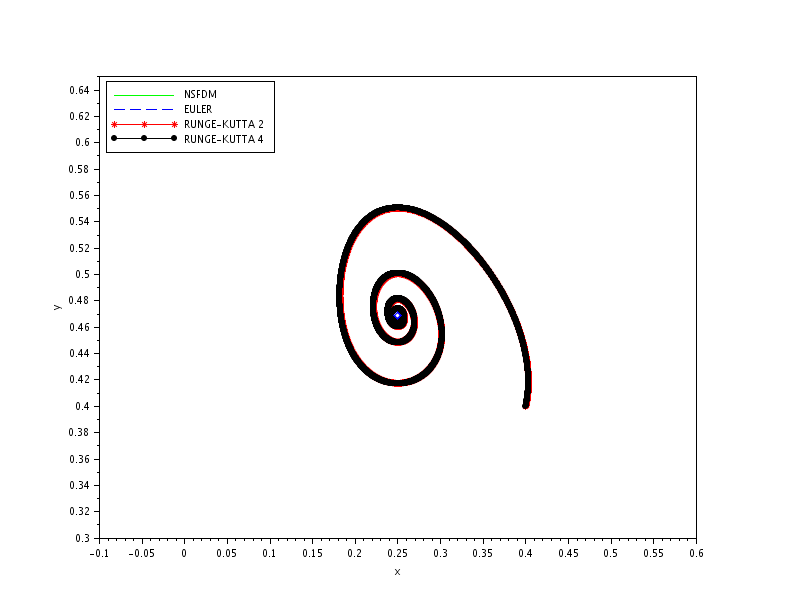}
                \caption{$h=0.01$}
        \end{subfigure}%
        ~ %add desired spacing between images, e. g. ~, \quad, \qquad, \hfill etc.
        \begin{subfigure}[b]{0.45\textwidth}
                \includegraphics[width=\textwidth]{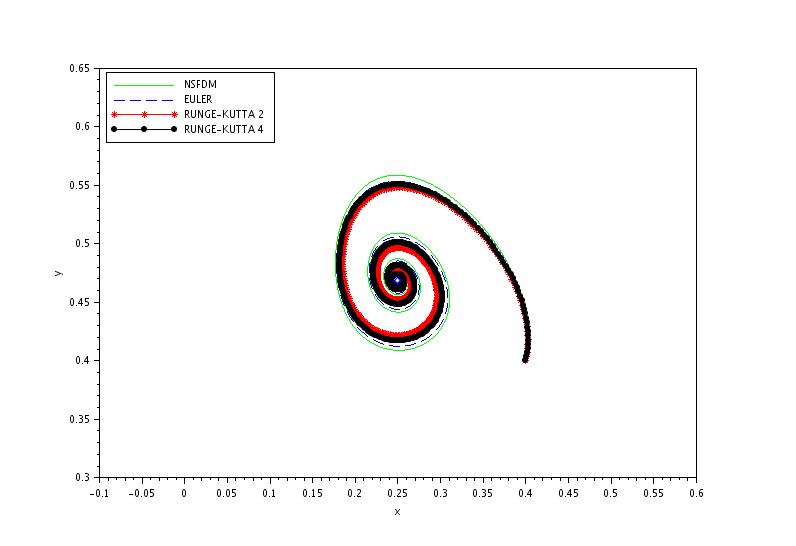}
                \caption{$h=0.1$}
        \end{subfigure}
                		
        \begin{subfigure}[b]{0.45\textwidth}
                \includegraphics[width=\textwidth]{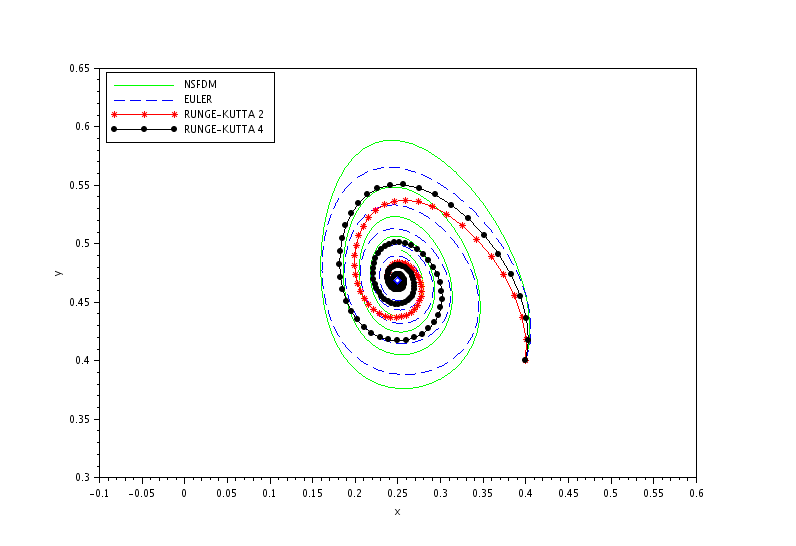}
                \caption{$h=0.5$}
        \end{subfigure}
        ~
        \begin{subfigure}[b]{0.45\textwidth}
                \includegraphics[width=\textwidth]{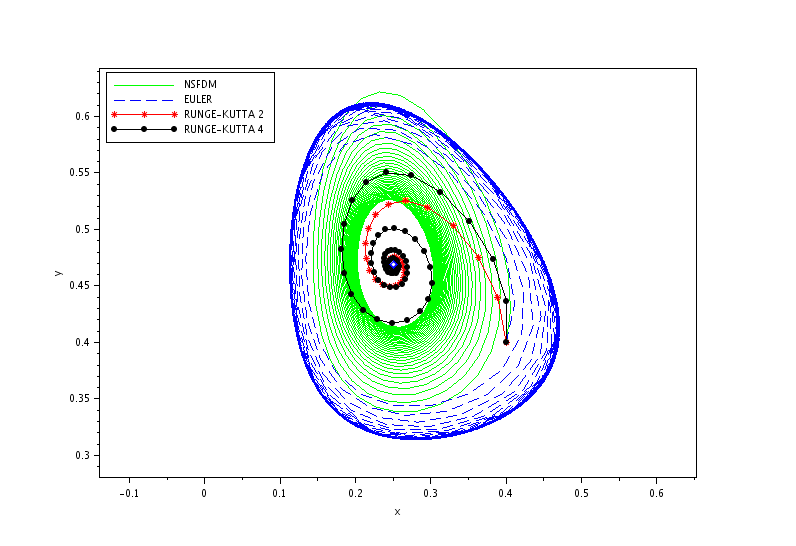}
                \caption{$h=1$}
        \end{subfigure}

        \begin{subfigure}[b]{0.45\textwidth}
                \includegraphics[width=\textwidth]{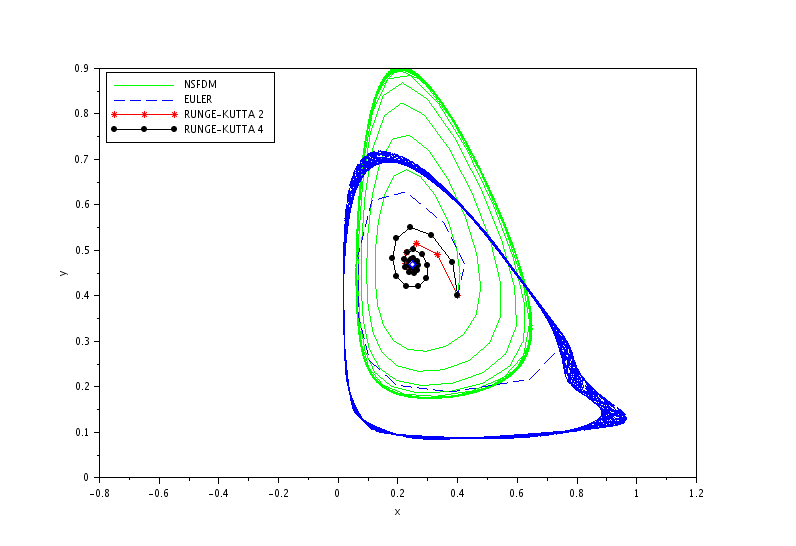}
                \caption{$h=2$}
        \end{subfigure}
        ~
        \begin{subfigure}[b]{0.45\textwidth}
                \includegraphics[width=\textwidth]{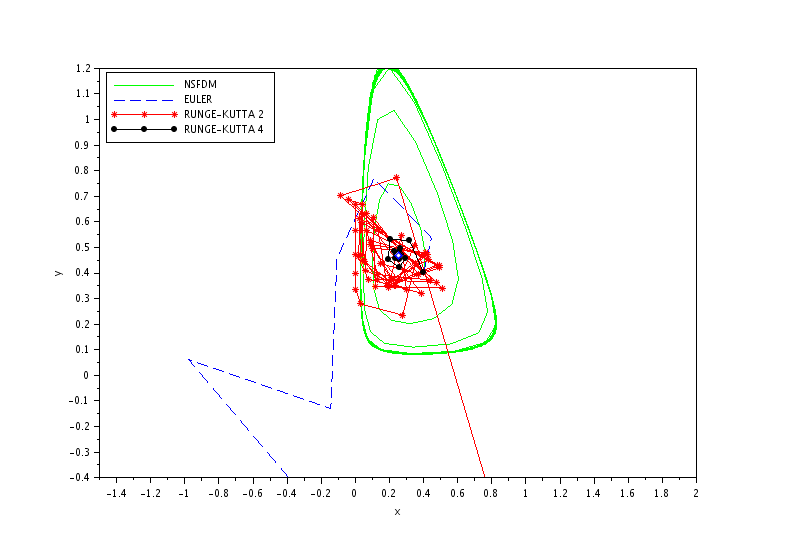}
                \caption{$h=4$}
        \end{subfigure}
        \caption{Numerical simulations of Example 4 with $x_0=0.4$, $y_0=0.4$.}\label{fig_ex4}
\end{figure}

\subsection{Invariance/positivity artefacts}

By construction, the NSFD respects the positivity of the systems unconditionally with respect to the time increment $h$ by Lemma \ref{lempos}. However, this is not the case for the classical numerical scheme : 
\begin{itemize}
\item Figure \ref{fig_ex1} shows that the Euler and the Runge-Kutta methods of order 2 or 4 do not respect the positivity property. 
\item Figure \ref{figex2} shows that the Euler method does not respect the positivity property for $h$ from $h=0.5$ to $h=0.1$. In the same way, the Runge-Kutta of order 4 does not respect positivity for $h$ from $h=0.5$ to $h=0.2$.
\end{itemize}

\section{Conclusion}

The nonstandard scheme studied in this paper generalizes previous results obtained by D.T. Dimitrov and H.V. Kojouharov in a series of papers \cite{dimitrov2006}, \cite{dimitrov2007} and \cite{dimitrov2008}. The convergence is proved as well as the fact that the scheme preserves the fixed points and their stability nature and also the positivity. The main advantages of this scheme are illustrated via numerical examples. These simulations show at least two things : 
\begin{itemize}
\item First, most of the time the non-standard scheme behaves better or equivalently to a Runge-Kutta methode of order 4. The algorithmic complexity of the non-standard scheme being comparable to the Euler scheme, the gain in term of computation is very huge. 

\item Second, the respect of dynamical constraint lead to a scheme which gives the good dynamical behaviour even for large time increment. This gives also a very big computational advantage.
\end{itemize}

The positivity or more generally domain invariance is an important issue for many applications, in particular in biology, and provide the first test to select models. This question is fundamental when one is dealing with stochastic differential equations (see \cite{cps1} and \cite{cps2}) as simulations are used to validate a given model. An interesting issue is then to develop {\it stochastic qualitative dynamical numerical scheme} for stochastic differential equations. A first step in this direction has been made by F. Pierret \cite{pierret} by the construction of a non-standard Euler-Murayama scheme.

\begin{appendix}

\section{Proof of Lemmas \ref{stabilityO} and \ref{stability-E1-E2}}

\subsection{Proof of Lemma \ref{stabilityO}} 
\label{proof-stabilityO}

The Jacobian matrix of $\phi$ is given by 
\begin{equation}
\label{jacobian}
D\varphi(x,y) = \begin{pmatrix}
\left(f_+-f_-\right)(x,y)+x\left(\partial_x f_+ - \partial_x f_- \right)(x,y) & x\left(\partial_y f_+ - \partial_y f_- \right)(x,y) \\ 
y\left(\partial_x g_+ - \partial_x g_- \right)(x,y) & \left(g_+-g_-\right)(x,y) +y\left(\partial_y g_+ - \partial_y g_- \right)(x,y)
\end{pmatrix}.
\end{equation}

At the origin the Jacobian matrix reduces to 
\begin{equation*}
D\varphi(0,0) = \begin{pmatrix}
\left(f_+-f_-\right)(0,0) & 0) \\ 
0 & \left(g_+-g_-\right)(0,0)
\end{pmatrix} .
\end{equation*}
The eigenvalues are then given by $\lambda^0_1 = \left(f_+-f_-\right)(0,0)$ and $\lambda^0_2=\left(g_+-g_-\right)(0,0)$.

\subsection{Proof of Lemma \ref{stability-E1-E2}}
\label{proof-E1-E2}

Using equation (\ref{jacobian}), we deduce that the Jacobian of $\phi$ at an equilibrium point of the family $E_1$ is given by
\begin{equation}
D\varphi(x_\sharp,0) = \begin{pmatrix}
x_\sharp\left(\partial_x f_+ - \partial_x f_- \right)(x_\sharp,0) & x_\sharp\left(\partial_y f_+ - \partial_y f_- \right)(x_\sharp,0) \\ 
0 & \left(g_+-g_-\right)(x_\sharp,0)
\end{pmatrix}
.
\end{equation}
The eigenvalues are easily obtained as $\lambda^1_1 = x_\sharp\left(\partial_x f_+ - \partial_x f_- \right)(x_\sharp,0)$ and $\lambda^1_2=\left(g_+-g_-\right)(x_\sharp,0)$.\\ 

In the same way, the Jacobian at an equilibrium point of the family $E_2$ has eigenvalues $\lambda^2_1=f_+(0,y_\sharp)-f_-(0,y_\sharp)$ and $\lambda^2_2=y_\sharp\left(\partial_y g_+ - \partial_y g_- \right)(0,y_\sharp)$. 

\section{Proof of Theorem \ref{thmconv}}
\label{proof-convergence}

As usual in the study of numerical algorithm (see \cite{dema},Chap.VIII,p.226-228), we prove {\it consistency} and {\it stability} of the NSFD scheme and then {\it convergence} (see \cite{dema},Corollaire,p.227).\\

The NSFD sheme is given by
\begin{align*} 
x_{k+1} &= x_k \left( \frac{1 + h f_+(x_k,y_k)}{1 + h f_-(x_k,y_k)} \right), \\
y_{k+1} &= y_k \left( \frac{1 + h g_+(x_k,y_k)}{1 + h g_-(x_k,y_k)} \right).
\end{align*}
A Taylor expansion with remainder of each component gives
\begin{align*}
x(t_{k+1}) &= x(t_k) + h\frac{dx}{dt}(t_k)  + \frac{1}{2} h^2 \frac{d^2 x}{dt^2}(t_k + \theta_x h)  = x(t_k) + h\varphi^1(x(t_k), y(t_k)) + \tau^1_{k} \\
y(t_{k+1}) &= y(t_k) + h\frac{dy}{dt}(t_k)  + \frac{1}{2} h^2 \frac{d^2 y}{dt^2}(t_k + \theta_y h)  = y(t_k) + h\varphi^2(x(t_k), y(t_k)) + \tau^2_{k}
\end{align*}
for some real $\theta_x$ and $\theta_y$ between $0$ and $1$. This defines the {\it local truncation error} $\tau_k$. 

Let $X_k=(x_k,y_k)$ and $X(t_k)=(x(t_k),y(t_k))$ for all $k\ge 0$. Subtraction of the NSFD scheme and Taylor expansions gives a difference equation for the error

\begin{equation*}
e_{k+1} = X_{k+1} - X(t_{k+1}) = e_k +h \begin{pmatrix} x_k \left(\frac{f_+-f_-}{1+hf_-}\right)(X_k) - x(t_k) \left(f_+-f_-\right)(X(t_k)) \\ y_k \left(\frac{g_+-g_-}{1+hg_-}\right)(X_k)  - y(t_k) \left(g_+-g_-\right)(X(t_k)) \end{pmatrix}- \tau_n
\end{equation*}

By hypothesis the first component of $\varphi$ is Lipschitz with constant $L_1$ and the second component is also Lipschitz with constant $L_2$. We denote by $L$ the maximum value between $L_1$ and $L_2$. \\

As $1+hf_-(x,y)\geq 1$ and $1+hg_-(x,y)\geq 1$ for all $x,y \geq 0$ and assume that the local truncation error satisfies a bound $\|\tau_k\|_\infty \leq \tau $ for all $k$ then
\begin{equation}
\|e_{k+1}\|_\infty \le \|e_k\|_\infty + hL\|e_k\|_\infty +\tau .
\end{equation}
Using the {\it discrete Gronwall Lemma} (see \cite{dema},p.333) we obtain
\begin{equation}
\|e_k\|_\infty \leq  e^{LT}\|e_0\|_\infty + \frac {e^{LT}-1}{LT} k\tau .
\end{equation}
We then have {\it stability}. \\

The local truncation error is equal to
\begin{equation}
\|\tau_k\|_\infty =\sup \left(\left|\frac{1}{2} h^2 \frac{d^2 x}{dt^2}(t_k + \theta_x h)\right|,\left|\frac{1}{2} h^2 \frac{d^2 y}{dt^2}(t_k + \theta_y h)\right|\right) \leq \frac{1}{2}Mh^2 .
\end{equation}
As $\varphi$ is $C^1([t_0,T],\mathbb{R}^2)$, we deduce that the solution is $C^2([t_0,T],\mathbb{R}^2)$ with bounded derivatives over $[t_0,T]$. As a consequence, we obtain 
\begin{equation}
\|\tau_k\|_\infty \leq \frac{1}{2}Mh^2 .
\end{equation}
We deduce that the NFSD scheme is {\em consistent}. \\

Consistency gives a local bound on $\tau$ and stability allows us to conclude {\em convergence}:

\begin{equation*}
\|X_k - X(t_k)\|_\infty \leq e^{LT} \|X_0-X(t_0)\| + \frac{e^{LT}-1}{LT} \frac{T}{2}Mh  \le O(\|X_0-X(t_0)\|_\infty) + O(h)
\end{equation*}

\section{Proof of Theorem \ref{thmstab1}}

The Jacobian matrix is given by
\begin{align*}
& D\varphi_{NS,h}(x,y) = \\
&\begin{pmatrix}
\frac{1 + h f_+(x,y)}{1 + h f_-(x,y)} + hx \left(\frac{(1+hf_-)\partial_x f_+ - (1+hf_+)\partial_x f_-}{(1+hf_-)^2}\right)(x,y) & hx \left(\frac{(1+hf_-)\partial_y f_+ - (1+hf_+)\partial_y f_-}{(1+hf_-)^2}\right)(x,y) \\ 
hy \left(\frac{(1+hg_-)\partial_x g_+ - (1+hg_+)\partial_x f_-}{(1+hg_-)^2}\right)(x,y) & \frac{1 + h g_+(x,y)}{1 + h g_-(x,y)} + hy \left(\frac{(1+hg_-)\partial_y g_+ - (1+hg_+)\partial_y f_-}{(1+hg_-)^2}\right)(x,y)
\end{pmatrix} .
\end{align*} \\

At the origin, the Jacobian reduces to
\begin{equation*}
D\varphi_{NS,h}(0,0) = \begin{pmatrix}
\frac{1 + h f_+(0,0)}{1 + h f_-(0,0)} & 0 \\ 
0 & \frac{1 + h g_+(0,0)}{1 + h g_-(0,0)}
\end{pmatrix}
\end{equation*}

has eigenvalues 

\begin{equation*}
\gamma^0_1=\frac{1 + h f_+(0,0)}{1 + h f_-(0,0)} \quad \text{and} \quad \gamma^0_2=\frac{1 + h g_+(0,0)}{1 + h g_-(0,0)} \ .
\end{equation*}

The stability conditions $(S_{E_0})$ imply $f_+(0,0)<f_-(0,0)$ and $g_+(0,0)<g_-(0,0)$ then

\begin{equation*}
0<\frac{1 + h f_+(0,0)}{1 + h f_-(0,0)}<1, \quad 0<\frac{1 + h g_+(0,0)}{1 + h g_-(0,0)}<1,
\end{equation*}
for arbitrary $h$. The origin is then stable for arbitrary $h$. \\

At $E_1$ the Jacobian
\begin{equation*}
D\varphi_{NS,h}(x_\sharp,0) = \begin{pmatrix}
1 + hx_\sharp \left(\frac{\partial_x f_+ - \partial_x f_-}{1+hf_+}\right)(x_\sharp,0) & hx_\sharp \left(\frac{\partial_y f_+ - \partial_y f_-}{1+hf_+}\right)(x_\sharp,0) \\ 
0 & \frac{1 + h g_+(x_\sharp,0)}{1 + h g_-(x_\sharp,0)}
\end{pmatrix}
\end{equation*}

has eigenvalues 

\begin{equation*}
\gamma^1_1=1 + hx_\sharp \left(\frac{\partial_x f_+ - \partial_x f_-}{1+hf_+}\right)(x_\sharp,0) \quad \text{and} \quad \gamma^1_2=\frac{1 + h g_+(x_\sharp,0)}{1 + h g_-(x_\sharp,0)} \ .
\end{equation*}

The stability conditions $(S_{E_1})$ imply $x_\sharp\left(\partial_x f_+ - \partial_x f_- \right)(x_\sharp,0)<0$ and $g_+(x_\sharp,0)<g_-(x_\sharp,0)$ then

\begin{equation*}
1 + hx_\sharp \left(\frac{\partial_x f_+ - \partial_x f_-}{1+hf_+}\right)(x_\sharp,0)<1, \quad 0<\frac{1 + h g_+(0,0)}{1 + h g_-(0,0)}<1,
\end{equation*}
for arbitrary $h$ that is to say the fixed point $E_1$ is stable for arbitrary $h$.\\

In a same way we obtain that $E_2$ is a stable fixed points for arbitrary $h$.

\section{Proof of Theorem \ref{thmstab2}}

The proof relies on the following classical result :

\begin{lemma}
Roots of the quadratic equation $\gamma^2-\alpha\gamma+\beta=0$ satisfy $|\gamma_i|<1$ for $i=1,2$ if and only if the following conditions hold :
\begin{enumerate}
\item[(a)] $1+\alpha+\beta >0$
\item[(b)] $1-\alpha+\beta >0$
\item[(c)] $\beta<1$
\end{enumerate}
\label{lemma_roots}
\end{lemma}

At an equilibrium point $(x_\star,y_\star)$ of type $E_3$ the Jacobian matrix is given by 
\begin{equation}
D\varphi_{NS,h}(x_\star,y_\star) = \begin{pmatrix}
 1+hx_\star\left(\frac{\partial_x f_+ - \partial_x f_- }{1+hf_+}\right)(x_\star,y_\star) & hx_\star\left(\frac{\partial_y f_+ - \partial_y f_- }{1+hf_+}\right)(x_\star,y_\star) \\
hy_\star\left(\frac{\partial_x g_+ - \partial_x g_- }{1+hg_+}\right)(x_\star,y_\star) & 1+hy_\star\left(\frac{\partial_y g_+ - \partial_y g_- }{1+hg_+}\right)(x_\star,y_\star) \\
\end{pmatrix}
\end{equation}
The trace denoted by $T_{\phi_{NS,h}}$ is given by 
\begin{equation*}
T_{\phi_{NS,h}} = 2 + h\left( x_\star\left(\frac{\partial_x f_+ - \partial_x f_- }{1+hf_+}\right)(x_\star,y_\star) + y_\star\left(\frac{\partial_y g_+ - \partial_y g_- }{1+hg_+}\right)(x_\star,y_\star) \right) .
\end{equation*}
and its determinant denoted by $D_{\phi_{NS,h}}$ is equal to
\begin{align*}
D_{\phi_{NS,h}} = & 1 + h\left( x_\star\left(\frac{\partial_x f_+ - \partial_x f_- }{1+hf_+}\right)(x_\star,y_\star) + y_\star\left(\frac{\partial_y g_+ - \partial_y g_- }{1+hg_+}\right)(x_\star,y_\star) \right) \\
& + h^2 \frac{Det\left(D\varphi(x_\star,y_\star)\right)}{(1+hf_+)(1+hg_+)(x_\star,y_\star)}
\end{align*}
We verify that 
\begin{equation}
\label{mainequa}
1-T_{\phi_{NS,h}} +D_{\phi_{NS,h}} = h^2 \frac{D}{(1+hf_+)(1+hg_+)(x_\star,y_\star)} .
\end{equation}

The eigenvalues of $D\varphi_{NS,h}(x_\star,y_\star)$ are the roots of the quadratic equation $\gamma^2-T_{\phi_{NS,h}} \gamma+ D_{\phi_{NS,h}}=0$. By Lemma \ref*{lemma_roots} we preserve stability if and only if conditions (a), (b) and (c) are satisfied. We begin with conditions (b) and (c). \\

{\bf Condition $(b)$} : We have for all $h>0$ that $(1+hf_+)(1+hg_+)(x_\star,y_\star) >0$ by definition of $f_+$ and $g_+$. As a consequence, condition (b) is equivalent to $D>0$. By assumption, the point $(x_\star,y_\star)$ is a stable point of (E) so that $D>0$. Condition (b) is then satisfied for all $h>0$.\\

{\bf Condition $(c)$} : We have 
\begin{equation}
D_{\phi_{NS,h}} = 1+\di\frac{h}{(1+hf_+)(1+hg_+)(x_\star,y_\star)} 
\left [
T+h(D+C)
\right ] ,
\end{equation}
where $C$ is given by 
\begin{equation}
C= x_\star\left[(\partial_x f_+ - \partial_x f_-)g_+\right](x_\star,y_\star)+\left[y_\star(\partial_y g_+ - \partial_y g_-)f_+\right](x_\star,y_\star) .
\end{equation}

The quantity $h/ [(1+hf_+)(1+hg_+)](x_\star,y_\star)$ is always positive for $h>0$. Moreover, as the equilibrium point $(x_\star,y_\star)$ is stable, we have $T<0$. As a consequence, the condition $D_{\phi_{NS,h}} <1$ is satisfied for arbitrary $h>0$ or $h$ sufficiently small depending on the sign of $D+C$. Precisely, we must have $T+h(D+C) <0$. We have $D>0$ by the stability assumption but no information on the sign of $C$. If $D+C\leq 0$ the condition is fulfilled for all $h>0$ and if $D+C>0$ we must have 
\begin{equation}
h<- \di\frac{T}{D+C} .
\end{equation}

{\bf Condition $(a)$} : We have 
\begin{equation}
1+T_{\phi_{NS,h}} +D_{\phi_{NS,h}} =4 +\di\frac{2h}{(1+hf_+)(1+hg_+)](x_\star,y_\star)} 
\left [
T+h(C+D) 
\right ] 
,
\end{equation}
using the previous notations.\\

By condition (c), we have $T+h(C+D) <0$. If all the cases, we must have $h$ sufficiently small in order to ensure condition (a). Indeed, if $T+h(C+D)$ is strictly negative unconditionally on $h>0$ then condition (a) is satisfied for $h< -2 \di\frac{(1+hf_+)(1+hg_+)](x_\star,y_\star)}{T}$. If $T+h(C+D) <\alpha <0$ for $h<h_0$ where $\alpha$ does not depend on $h$ then $h< -2 \di\frac{(1+hf_+)(1+hg_+)](x_\star,y_\star)}{\alpha}$.\\

This concludes the proof.

\end{appendix}

\end{document}